\documentclass[preprint,12pt]{elsarticle}
\usepackage{cmap}
\usepackage[T1]{fontenc}
\usepackage[utf8]{inputenc}
\usepackage{amsfonts, amsmath, amscd}
\usepackage{amssymb}
\usepackage{amsthm}
\usepackage[all,cmtip]{xy}
\usepackage{comment}
\usepackage{mathrsfs}
\usepackage{xcolor}
\usepackage{enumitem}
\usepackage[pdftex,unicode]{hyperref}

\newtheorem{proposition}{Proposition}
\newtheorem{theorem}{Theorem}
\newtheorem*{theorem*}{Theorem}
\newtheorem{lemma}{Lemma}
\newtheorem*{lemma*}{Lemma}
\newtheorem{cor}{Corollary}

\theoremstyle{definition}
\newtheorem{example}{Example}

\newtheorem{question}{Question}
\newtheorem{problem}{Problem}

\theoremstyle{remark}
\newtheorem{note}{Remark}
\newtheorem*{note*}{Remark}

\def\N{\mathbb N}

\def\R{\mathbb R}

\def\om{\omega}

\def\supp{\operatorname{supp}}

\def\dim{\operatorname{dim}}
\def\al{\alpha}

\def\la{\lambda}

\def\de{\delta}

\def\si{\sigma}
\def\ep{\varepsilon}
\def\te{\theta}

\def\cP{\mathcal{P}}

\def\cW{\mathcal{W}}

\def\cL{\mathcal{L}}
\def\cL{\mathcal{L}}
\def\cS{\mathcal{S}}
\def\cW{\mathcal{W}}
\def\cZ{\mathcal{Z}}

\def\fZ{\mathfrak Z}

\def\dd{\mathscr{D}}
\def\sA{\mathscr{A}}
\def\sB{\mathscr{B}}

\def\rarr{\Rightarrow}
\def\larr{\Leftarrow}
\def\lrarr{\Leftrightarrow}

\def\cl#1{\overline{#1}}

\def\env#1{\langle #1\rangle}

\def\zz#1{#1^{-1}(0)}
\def\cc#1{\bar{#1}}

\begin{document}

\begin{frontmatter}

\title{When are metrizable linear subspaces of topological vector spaces finite-dimensional?}

\author[msu1]{V.\,V.~Maksimov}
\ead{vladislav.maksimov@math.msu.ru}
\author[msu2,smbu]{E.\,A.~Reznichenko\corref{cor1}}
\ead{erezn@inbox.ru}
\cortext[cor1]{Corresponding author}

\affiliation[msu1]{organization={Department of General Problems of Control, Faculty of Mechanics and Mathematics, M.\,V.~Lomonosov Moscow State University},
            addressline={Leninskie Gory 1},
            city={Moscow},
            postcode={119991},
            country={Russia}}
\affiliation[msu2]{organization={Department of General Topology and Geometry, Faculty of Mechanics and Mathematics, M.\,V.~Lomonosov Moscow State University},
            addressline={Leninskie Gory 1},
            city={Moscow},
            postcode={119991},
            country={Russia}}
\affiliation[smbu]{organization={Shenzhen MSU-BIT University},
            city={Shenzhen},
            country={China}}

\begin{abstract}
We study conditions that imply finite-dimensionality of metrizable or Fr\'echet--Urysohn linear subspaces of topological vector spaces (TVS). We consider the classes $\dd_M$ and $\dd_{FU}$ of TVS: a space $L$ belongs to $\dd_M$ (respectively, $\dd_{FU}$) if every metrizable (respectively, Fr\'echet--Urysohn) linear subspace of $L$ is finite-dimensional.
Although the classes of metrizable spaces and of Fr\'echet--Urysohn spaces differ substantially, the classes $\dd_M$ and $\dd_{FU}$ are much closer to each other: in the Hru\v{s}\'ak--Ramos-Garc\'{\i}a model, in which Malykhin's problem is solved, $\dd_M=\dd_{FU}$, whereas under $\mathrm{MA+\lnot CH}$, as shown in the paper, $\dd_M\neq\dd_{FU}$.
For a locally convex space (LCS) $L$ we denote by $L_w$ the space $L$ with its weak topology and by $L'_{w^\ast}$ the dual space with the weak${}^\ast$ topology. This paper investigates which LCS $L$ have the property that $L_w$ and $L'_{w^\ast}$ belong to $\dd_M$ or $\dd_{FU}$.
Although the classes $\dd_M$ and $\dd_{FU}$ are close, different methods are used to study them: topological-algebraic ones for $\dd_M$ (a criterion is obtained) and methods of measure theory and probability for $\dd_{FU}$ (a sufficient condition is obtained).
We prove that if an LCS $L$ admits a continuous norm, then $L_w\in\dd_{FU}$; in particular, if $L$ admits a continuous norm and $L_w$ is Fr\'echet--Urysohn, then $L$ is finite-dimensional. If the topology of an LCS $L$ is coarser than the topology induced by some complete norm on $L$ (for example, if $L$ is Banach), then $L'_{w^\ast}\in\dd_{FU}$; completeness of the norm is essential here.
The methods of the paper also apply to the spaces of real-analytic functions and to the space of polynomials with the topology of pointwise convergence: these also belong to $\dd_{FU}$.
\end{abstract}


\end{frontmatter}

\section{Introduction}\label{sec:intro}

One of the directions of research in functional analysis is to find conditions on a topological vector space (TVS) that force it to be finite-dimensional.
Fundamental results of this kind are Riesz's theorem that a locally compact TVS is finite-dimensional \cite{Riesz1918}, and its measure-theoretic counterpart: a TVS carrying a nonzero translation-invariant Radon measure (Haar measure) is finite-dimensional (this follows from Weil's theorem \cite{Weil1940} and Riesz's theorem, see \cite[Theorem 5.3]{Yamasaki1985}).

This is closely related to the following general problem.

\begin{problem}
Find classes $\sA$ and $\sB$ of TVS such that every subspace $S\subset L$ belonging to $\sB$ is finite-dimensional whenever $L\in\sA$.
\end{problem}

In the Schwartz space every normable subspace is finite-dimensional.\footnote{This follows from the fact that the Schwartz space is a Montel space: the unit ball of a normable subspace $N$ is bounded in $N$, hence bounded in the Schwartz space, hence relatively compact there and, in particular, totally bounded in $N$; but a normed space with a totally bounded unit ball is finite-dimensional (Riesz's lemma).} It is shown in \cite[Theorems~7.9, 7.10]{LRS-TVS} that in the free locally convex space $L(X)$ over a Tychonoff space $X$, in its weak counterpart $L_p(X)$ (the space $L(X)$ with its weak topology), and in the free TVS $V(X)$, every $\varkappa$-Fr\'echet--Urysohn or Baire subspace is finite-dimensional. Related results on $\varkappa$-Fr\'echet--Urysohn and Baire subgroups of free topological groups are obtained in \cite{LRS-groups}.

Given a class $\sB$ of TVS, let $\dd_\sB$ denote the largest class $\sA$ for which the conclusion of Problem~1 holds. Thus, $L\in\dd_\sB$ if and only if every subspace of $L$ belonging to $\sB$ is finite-dimensional. In this paper we study two such classes: $\dd_M$ corresponds to the class $\sB$ of metrizable TVS, and $\dd_{FU}$ to the class $\sB$ of Fr\'echet--Urysohn TVS:
\begin{enumerate}
\item[$\dd_M$:] $L\in \dd_M$ if every metrizable subspace of $L$ is finite-dimensional;
\item[$\dd_{FU}$:] $L\in \dd_{FU}$ if every Fr\'echet--Urysohn subspace of $L$ is finite-dimensional.
\end{enumerate}
Clearly, $\dd_{FU}\subset \dd_{M}$.

Malykhin's problem was solved in \cite{HR2014}: a model of set theory (which we call the Hru\v{s}\'ak--Ramos-Garc\'{\i}a model) was constructed in which every separable Fr\'echet--Urysohn topological group is metrizable. Since every infinite-dimensional TVS contains a countable-dimensional infinite-dimensional subspace (which is then separable), we have $\dd_M=\dd_{FU}$ in the Hru\v{s}\'ak--Ramos-Garc\'{\i}a model. On the other hand, under Martin's Axiom and the negation of the Continuum Hypothesis ($\mathrm{MA+\lnot CH}$) we construct (see Example~\ref{ex:pol:FU}) a countable-dimensional nonmetrizable Fr\'echet--Urysohn space $\cP_X$ (the space of polynomial functions on a set $X\subset\R$ with $|X|=\om_1$, see Section~\ref{sec:pol}) which belongs to $\dd_M$; it is infinite-dimensional and Fr\'echet--Urysohn, so it does not belong to $\dd_{FU}$. Thus $\dd_M\neq\dd_{FU}$ under $\mathrm{MA+\lnot CH}$. Consequently:

\begin{theorem}
\label{th:ddM-ddFU}
If $\mathrm{ZFC}$ is consistent, then the equality $\dd_M=\dd_{FU}$ can be neither proved nor refuted in $\mathrm{ZFC}$:
\begin{itemize}
\item
$\dd_M=\dd_{FU}$ in the Hru\v{s}\'ak--Ramos-Garc\'{\i}a model;
\item
$\dd_M\neq\dd_{FU}$ under $\mathrm{MA+\lnot CH}$.
\end{itemize}
\end{theorem}

We consider the following variant of the finite-dimensionality problem. Let $L$ be a locally convex space and $L_w$ the space $L$ with its weak topology. Which conditions on $L$ and $L_w$ imply that $L$ is finite-dimensional?

Let $L$ be a normable space. A classical result going back to Banach \cite{Banach1932} says that if the topologies of $L$ and $L_w$ coincide, then $L$ is finite-dimensional.\footnote{A stronger statement holds: if $L_w$ is normable (and $L$ is an arbitrary LCS), then $L$ is finite-dimensional. Indeed, a normable $L_w$ has a bounded neighbourhood of zero $U$; it contains a subspace $\bigcap_{i\le n}\ker f_i$ for some $f_1,\dots,f_n\in L'$. A bounded subspace is zero, so the map $x\mapsto (f_i(x))_{i\le n}$ is a linear injection $L\to\R^n$.}
This has been strengthened: if $L$ is infinite-dimensional, then $L_w$ is not metrizable (folklore, see \cite[p.~10]{Diestel}) and is not a Baire space \cite{Edgar1984}; if $L$ is Banach, then $L_w$ is not a $k$-space \cite[Proposition~6.4]{KM2017}, is not a Fr\'echet--Urysohn space \cite{SW1991}, and does not have the Pytkeev property \cite[Theorem~1.4]{GKZ2017}; moreover, $L_w$ does not have the Ascoli property and hence is not a $k_{\mathbb R}$-space \cite{GKP2016}. Furthermore, for a normable $L$ every $\varkappa$-Fr\'echet--Urysohn or Baire linear subspace of $L_w$ is finite-dimensional \cite{LRS-TVS}.

In this paper we study LCS with the weak topology that belong to $\dd_M$ and $\dd_{FU}$.
An LCS $L_w$ with its weak topology is realized as a subspace of the space of functions $\R^X$ with the topology of pointwise convergence (that is, the product topology on $\R^X$), where $X=L'$ is the dual space.

For subspaces $L\subset \R^X$ we obtain a criterion for $L\in\dd_M$ (Theorem~\ref{th:M}) and a sufficient condition for $L\in\dd_{FU}$ (Theorem~\ref{th:FU}).

We prove that if $L$ is the space of real-analytic functions on a domain $G\subset \R^n$ with the topology of pointwise convergence, then $L\in\dd_{FU}$ (Proposition~\ref{prop:a-FU}).

If an LCS $V$ admits a continuous norm, then $V_w\in\dd_{FU}$ (Theorem~\ref{th:wlcs:norm-dd-FU}).
In particular, if an LCS $V$ admits a continuous norm and $V_w$ is Fr\'echet--Urysohn, then $V$ is finite-dimensional (Corollary~\ref{cor:wlcs:norm-dd-FU}).
For Banach $V$ this follows from the results of \cite[p.~280]{SW1991} (it is shown there that the weak topology of an infinite-dimensional Banach space is not a $k$-space, while Fr\'echet--Urysohn spaces are $k$-spaces); for normable $V$ it also follows from \cite[Corollary~5.4]{LRS-TVS}.
If the topology of an LCS $V$ is weaker than the topology of some complete norm (for example, $V$ is Banach), then $V'_{w^*}\in\dd_{FU}$ (Theorem~\ref{th:wstar-complete}); completeness of the norm is essential (Remark~\ref{note:wstar-countable}).

It is proved in \cite{LRS-TVS} that if $E$ is a Banach space, then the dual space $E'_{w^\ast}$ with the weak${}^\ast$ topology belongs to $\dd_{FU}$ (see \cite[Theorem~5.8]{LRS-TVS}).
Theorem~\ref{th:wstar-complete} extends this statement to LCS whose topology is weaker than the topology of some complete norm; the proof relies on measure-theoretic methods.

\section{Notation, definitions and preliminaries}\label{sec:pre}

By (topological) spaces we mean Tychonoff spaces. Locally convex spaces (LCS) are always assumed to be Hausdorff.

For a TVS $V$, a subset $L\subset V$ is called a \emph{subspace} if $L$ is a linear subset of $V$; it is endowed with the induced topology, with respect to which $L$ is also a TVS.

For $S\subset V$ we denote by $\env S$ the linear span of $S$.

A set $X\subset\R$ is said to have \emph{absolute} (or \emph{universal}) \emph{measure zero} if $\mu^*(X)=0$ for every continuous (atomless) finite Borel measure $\mu$ on $\R$ (that is, vanishing on singletons); equivalently, there is no nonzero continuous finite Borel measure on $X$ \cite[Theorem~8.2]{Bukovsky}.

A space $X$ has the \emph{property $(\gamma)$} \cite{GerlitsNagy1982} if for every $\om$-cover $\mathcal U$ of $X$ (that is, an open cover of $X$ not containing $X$ as an element such that every finite subset of $X$ is contained in some element of $\mathcal U$) there is a sequence of pairwise distinct elements $U_n\in\mathcal U$, $n\in\om$, such that every point of $X$ belongs to all $U_n$ except finitely many. A subset $X\subset\R$ has \emph{strong measure zero} if for every sequence $\ep_n>0$ there is a cover of $X$ by intervals $I_n$ with lengths $|I_n|<\ep_n$, $n\in\om$. A space $X$ has the \emph{Rothberger property} if for every sequence $(\mathcal U_n)_{n\in\om}$ of open covers of $X$ one can choose $U_n\in\mathcal U_n$ such that $\{U_n:n\in\om\}$ covers $X$. We shall need the following known facts.
\begin{itemize}
\item[(F1)] The space $C_p(X)$ is Fr\'echet--Urysohn if and only if $X$ has the property $(\gamma)$ \cite{GerlitsNagy1982}.
\item[(F2)] If $X\subset\R$ has the property $(\gamma)$, then $X$ has the Rothberger property \cite{GerlitsNagy1982} (see \cite[Theorem~8.109]{Bukovsky}), hence has strong measure zero \cite[Theorem~8.102]{Bukovsky}, and therefore is of absolute measure zero \cite[Theorem~8.11]{Bukovsky}. In particular, if $X\subset\R$ and $C_p(X)$ is a Fr\'echet--Urysohn space, then $X$ has strong measure zero and, therefore, is of absolute measure zero.
\item[(F3)] Under $\mathrm{MA}+\neg\mathrm{CH}$ we have $\om_1<\mathfrak p=\mathfrak c$, and every $X\subset\R$ of cardinality $<\mathfrak p$ (in particular, of cardinality $\om_1$) has the property $(\gamma)$ (Galvin--Miller \cite{GalvinMiller1984}; see \cite[Theorem~8.90, Corollary~9.4]{Bukovsky}).
\end{itemize}

A set is called \emph{countable} if it is equinumerous with $\om$, and \emph{at most countable} if it is finite or countable. A linear space is said to be \emph{countable-dimensional} if it has a finite or countable basis (that is, its algebraic dimension does not exceed $\aleph_0$); an infinite-dimensional countable-dimensional space has a countable basis. By a \emph{basis} of a linear space we always mean an algebraic (Hamel) basis, not a Schauder basis.

Let $(X,\cS)$ be a measurable space, where $\cS$ is a $\si$-algebra on $X$.
For a positive measure $\mu$ on $(X,\cS)$ we denote by $\cL^0(X,\mu)$ the set of $\mu$-measurable real-valued functions.

A measure $\mu$ on $X$ is called \emph{nonzero} if $\mu(X)>0$, and \emph{$\si$-finite} if there is a sequence $(X_n)_n$ of $\mu$-measurable subsets of $X$ such that $X=\bigcup_n X_n$ and $\mu(X_n)<\infty$ for all $n$.
Two measures on $(X,\cS)$ are called \emph{equivalent} if for every $M\in \cS$, $\mu(M)=0$ if and only if $\nu(M)=0$.

\begin{proposition}[{\cite[Lemma 4.2.1]{Kon2017}}]
\label{prop:pre:eq-me}
Every nonzero $\sigma$-finite measure is equivalent to a probability measure.
\end{proposition}

Let $X$ be a set and $(M_n)_n$ a sequence of subsets of $X$. We denote the \emph{lower limit of a sequence of sets} by
\[
\liminf_{n\to\infty} M_n = \bigcup_{m=1}^\infty \bigcap_{n=m}^\infty M_n.
\]
Recall the definition of the lower limit of a sequence of numbers
\[
\liminf_{n\to\infty} x_n =  \sup_{n \geq 1}\,\inf_{m \geq n} x_m = \sup\,\{\, \inf\,\{\,x_m : m \geq n\,\} : n \geq 1\,\}.
\]

The following statement follows from Fatou's lemma, but we give a direct proof.

\begin{proposition}
\label{prop:pre:liminf-me}
Let $(X,\cS)$ be a measurable space, $\mu$ a measure on $X$ and $(M_n)_{n\in\N}$ a sequence of elements of $\cS$. Then
\[
\mu \bigl(\liminf_{n\rightarrow \infty }M_{n}\bigr)\le \liminf _{n\rightarrow \infty }\mu \bigl(M_{n}\bigr).
\]
\end{proposition}
\begin{proof}
Put $I_m=\bigcap_{n=m}^\infty M_n$. Then
\[
\liminf_{n\to\infty} M_n = \bigcup_{m=1}^\infty I_m,
\]
and the sequence of sets $(I_m)$ is increasing. Hence
\[
\mu \bigl(\liminf_{n\to\infty}M_{n}\bigr)=\sup_{m\geq 1}\mu(I_m).
\]
The required inequality follows from the fact that $I_m\subset M_n$ for $n\geq m$, and hence $\mu(I_m)\leq \inf_{n\geq m}\mu(M_n)$.
\end{proof}

For $f\in \R^X$ we denote
\[
\supp f = \{ x\in X: f(x)\neq 0\}
\]
the \emph{support} of $f$.

An indexed family $(U_\al)_{\al\in A}$ of subsets of a set $X$ is called \emph{\mbox{point-finite}} if 
$|\{\al\in A: x\in U_\al\}|<\om$ for all $x\in X$ (repetitions are counted). Below, the notation $\{\supp u_n:n\in\om\}$ stands for the indexed family $(\supp u_n)_{n\in\om}$.

For an LCS $V$ we denote by $V'$ the dual space and put
\[
c\colon V\to V'',\ c(v)(f)=f(v)\ \text{ for }v\in V\text{ and }f\in V'
\]
the canonical embedding of $V$ into $V''$.
Since $V$ is Hausdorff, by the Hahn--Banach theorem $V'$ separates the points of $V$, so $c$ embeds $V_w$ into $\R^{V'}$.
We write $\cc v$ for $c(v)$.
In particular
\[ \zz {\cc v}  = \{\te\in V': \te(v)=0\}. \]

\section{The classes $\dd_M$ and $\dd_{FU}$ in function spaces}\label{sec:fun}

Let $X$ be a set and $L$ a subspace of $\R^X$. The sets
\[
W(M,\ep)=\{ f\in L: f(M) \subset (-\ep,+\ep)\},
\]
where $M$ is a finite subset of $X$ and $\ep>0$, form a base of the topology of $L$ at zero.
For $x\in X$ denote by
\[
\de_x: L\to \R,\ f \mapsto f(x)
\]
the evaluation functional at the point $x$.

\begin{lemma}
\label{l:L-fun}
Let $\te\in L'$ and $W(M,\ep)\subset \te^{-1}((-1,1))$ for some $\ep>0$ and a finite set $M=\{x_1,...,x_n\}\subset X$. Then $\te=\sum_{i=1}^n \la_i \de_{x_i}$ for some $\la_1,...,\la_n\in\R$.
\end{lemma}
\begin{proof}
Put $S=\{f\in L: f(M)=\{0\}\}$. If $f\in S$, then $k f\in W(M,\ep)$ for all $k\in\R$.
Hence $k \te(f)= \te(kf)\in (-1,1)$ for all $k\in\R$. It follows that $\te(f)=0$.
That is, $S\subset\ker\te$. Put $\pi_M\colon L\to\R^M$, $f\mapsto f|_M$. Since $\ker\pi_M=S\subset\ker\te$, there is a linear functional $\te'$ on $\pi_M(L)$ such that $\te=\te'\circ\pi_M$; extend $\te'$ to a linear functional on $\R^M$. Hence $\te=\sum_{i=1}^n \la_i \de_{x_i}$ for some $\la_1,...,\la_n\in\R$.
\end{proof}

This lemma implies the following statement.

\begin{lemma}
\label{l:L-dual}
The set $\{\de_x:x\in X\}$ generates $L'$, that is, every functional $\te\in L'$ can be represented in the form $\sum_{i=1}^n \la_i \de_{x_i}$.
\end{lemma}
\begin{proof}
Since $\te$ is continuous, $W(M,\ep)\subset \te^{-1}((-1,1))$ for some finite $M\subset X$ and $\ep>0$. It remains to apply Lemma~\ref{l:L-fun}.
\end{proof}

\begin{lemma}
\label{l:L-metr-1}
The following conditions are equivalent:
\begin{enumerate}
\item[\rm (i)]
the space $L$ is metrizable;
\item[\rm (ii)]
the set $\{\de_x:x\in T\}$ forms a basis of $L'$ for some at most countable $T\subset X$;
\item[\rm (iii)]
the space $L'$ is countable-dimensional.
\end{enumerate}
\end{lemma}
\begin{proof}
\hyphenpenalty=10000
(i)$\rarr$(ii)
Since $L$ is metrizable, it has a countable neighbourhood base at zero consisting of sets of the form $W(M_n,\ep_n)$, $n\in\om$. Put $Y=\bigcup_{n\in\om} M_n$. Let $\te\in L'$. Since $\te$ is continuous, $W(M,\ep)\subset \te^{-1}((-1,1))$ for some $\ep>0$ and a finite set $M=\{x_1,...,x_n\}\subset Y$. By Lemma~\ref{l:L-fun}, $\te=\sum_{i=1}^n \la_i \de_{x_i}$ for some $\la_1,...,\la_n\in\R$. 
Choose $T\subset Y$ such that $\{\de_x:x\in T\}$ is a maximal linearly independent system in $\{\de_x:x\in Y\}$.
Then $\{\de_x:x\in T\}$ is a basis of $\env{\{\de_x:x\in Y\}}$. As shown above, $L'=\env{\{\de_x:x\in Y\}}$. Hence $\{\de_x:x\in T\}$ is a basis of $L'$; the set $T$ is at most countable, since $Y$ is.

\hyphenpenalty=50
(ii)$\rarr$(iii) 
Obvious.

(iii)$\rarr$(i) 
The topology of $L$ is the weak topology with respect to the family $\{\de_x:x\in X\}\subset L'$, hence also with respect to any family of functionals whose linear span contains all $\de_x$, for example with respect to an at most countable basis $\{\te_n:n\in I\}$ of $L'$, where $I\subset\om$ (by Lemma~\ref{l:L-dual}, each $\de_x$ is a linear combination of the $\te_n$). Therefore $L$ embeds into $\R^I$ and is metrizable.
\end{proof}

\begin{lemma}
\label{l:L-metr}
\emergencystretch=1.5em Let $L$ be infinite-dimensional and countable-dimensional.
The space $L$ is metrizable if and only if there is a Hamel basis ${\{u_n:n\in\om\}\subset L}$ such that the family $\{\supp u_n:n\in \om\}$ is point-finite.
\end{lemma}
\begin{proof}
$(\rarr)$ 
By Lemma~\ref{l:L-metr-1} there is a countable set $T=\{t_n:n\in\om\}\subset X$ such that $\{\de_x:x\in T\}$ is a basis of $L'$. Since $L$ is infinite-dimensional and $L'$ separates the points of $L$ (for example, by the functionals $\de_x$), $L'$ is infinite-dimensional (otherwise, taking a basis $\te_1,\dots,\te_k$ of $L'$, we would obtain a linear injection $L\to\R^k$, $v\mapsto(\te_i(v))_{i\le k}$). Hence $T$ is infinite.
 
Let $\{v_n:n\in\om\}$ be a Hamel basis of $L$. 

Let $n\in\om$. 
Denote $T_n=\{t_0,...,t_{n-1}\}$. Put
\[
\pi_n\colon L\to \R^{T_n},\ f\mapsto f|_{T_n}.
\] 
Since $\{\de_{t_k}:k\in\om\}$ is linearly independent, $\pi_n(L)=\R^{T_n}$ (the linearly independent functionals $\de_{t_0},\dots,\de_{t_{n-1}}$ define a surjection $L\to\R^{n}$). 
Put
\[
m_n = \min \{m\in\om: \pi_n(\env{v_0,v_1,...,v_{m-1}})=\R^{T_n}\}.
\]
Clearly $m_n\geq n$ (since $\dim\R^{T_n}=n$).

The sequence $(m_n)$ is nondecreasing (the projection $\R^{T_{n+1}}\to\R^{T_n}$ is surjective). Let $m\in\om$. Put
\[
l_m = \max \{ k\in\om: m_k \leq m \}
\]
(the maximum exists, since $m_0=0\le m$ and $m_k\geq k$). Then $l_m\leq m_{l_m}\leq m$,
\[
\pi_{l_m}(\env{v_0,v_1,...,v_{m-1}})=\R^{T_{l_m}}
\]
and $l_m\to \infty$ as $m\to \infty$ (if $m\geq m_k$, then $l_m\geq k$).
Choose $w_m\in \env{v_0,v_1,...,v_{m-1}}$ such that $\pi_{l_m}(w_m)=\pi_{l_m}(v_m)$.
Put $u_m=v_m-w_m$. Then $\{u_n:n\in\om\}$ is a Hamel basis of $L$ and $\pi_{l_m}(u_m)=0$, that is, $u_m|_{T_{l_m}}=0$.

Let us show that the family $\{\supp u_n:n\in \om\}$ is point-finite.
Let $x\in X$. Since the set $\{\de_{t_n}:n\in\om\}$ generates $L'$, we have $\de_x\in \env{\{\de_t:t\in T_l\}}$ for some $l\in\om$.
Since $l_m\to \infty$, there is $N\in\om$ such that $l_m\geq l$ for $m\geq N$.

Let $m\geq N$.
Then $u_m|_{T_{l}}=0$. Since $\de_x\in \env{\{\de_t:t\in T_l\}}$, we get $u_m(x)=\de_x(u_m)=0$.

We obtain $|\{m\in \om: u_m(x)\neq 0\}|\leq N$, that is, the family $\{\supp u_n:n\in \om\}$ is point-finite.

$(\larr)$ 
Let $\{u_n:n\in\om\}\subset L$ be a Hamel basis of $L$ such that the family $\{\supp u_n:n\in \om\}$ is point-finite.
Put
\[
\psi\colon L'\to \R^\om, \ \te \mapsto (\te(u_n))_{n\in\om}.
\]
The map $\psi$ is linear and injective. Point-finiteness of the family of supports implies that
\[
\psi(\de_x)\in S =\{(x_n)_{n\in\om}: |\{n\in\om:x_n\neq 0\}|<\om\}
\]
for every $x\in X$. 
By Lemma~\ref{l:L-dual}, $\psi(L')\subset S$.
The space $S$ is the $\si$-product in $\R^\om$ and is countable-dimensional.
Since $\psi$ is linear and injective, the space $L'$ is linearly isomorphic to a subspace of $S$. Hence $L'$ is countable-dimensional.
\end{proof}

\begin{lemma}
\label{l:L-pfs}
If there is a sequence $\{u_n:n\in\om\}\subset L\setminus\{0\}$ such that the family $\{\supp u_n:n\in \om\}$ is point-finite, then $L$ is infinite-dimensional.
\end{lemma}
\begin{proof}
Suppose the contrary. Let $\te_1,\te_2,\dots,\te_n$ be a basis of $L'$. By Lemma~\ref{l:L-dual} there is a finite $Y\subset X$ such that $\te_i\in K=\env{\{\de_x:x\in Y\}}$ for $i=1,...,n$. Then $L'=K$. Since the family $\{\supp u_n:n\in \om\}$ is point-finite and $Y$ is finite, $u_n|_Y=0$ for all but finitely many $n$; choose such an $m$. Then $\te(u_m)=0$ for all $\te\in K=L'$. But $L'$ separates the points of $L$ (for example, by the functionals $\de_x$), hence $u_m=0$, a contradiction.
\end{proof}

\begin{proposition}
\label{prop:M}
Let $X$ be a set and $L$ a subspace of $\R^X$. Then $L$ contains an infinite-dimensional metrizable subspace if and only if there is a sequence of nonzero functions $\{u_n:n\in\om\}\subset L$ such that the family $\{\supp u_n:n\in \om\}$ is point-finite.
\end{proposition}
\begin{proof}
$(\rarr)$ Let $S\subset L$ be an infinite-dimensional metrizable subspace. Take a countable-dimensional infinite-dimensional subspace $S_0\subset S$; it is metrizable. By Lemma~\ref{l:L-metr}, $S_0$ has a Hamel basis $\{u_n:n\in\om\}$ with a point-finite family of supports; $u_n\neq 0$.

$(\larr)$ By Lemma~\ref{l:L-pfs}, the space $\env{\{u_n:n\in\om\}}$ is infinite-dimensional. For some infinite $N\subset\om$ the family $\{u_n:n\in N\}$ is linearly independent. The family $\{\supp u_n:n\in N\}$ is point-finite, so by Lemma~\ref{l:L-metr} the countable-dimensional infinite-dimensional space $\env{\{u_n:n\in N\}}$ is metrizable.
\end{proof}

\begin{proposition}
\label{prop:FU}
Let $(X,\cS)$ be a measurable space, $\mu$ a nonzero $\si$-finite measure on $(X,\cS)$, and $L$ a Fr\'echet--Urysohn subspace of $\R^X$. If every nonzero function $f$ in $L$ is $\mu$-measurable and $\mu(\zz f)=0$, then $L$ is finite-dimensional.
\end{proposition}
\begin{proof}
Since every $\si$-finite measure is equivalent to a probability measure (Proposition~\ref{prop:pre:eq-me}), we may assume that $\mu$ is a probability measure.
We argue by contradiction and suppose that $L$ is infinite-dimensional. 
For a finite $F\subset X$ denote
\[
\pi_F\colon L \to \R^F,\ f\mapsto f|_{F}.
\]
Since $L$ is infinite-dimensional, $\R^F$ is finite-dimensional and $\pi_F$ is linear, the kernel $\ker \pi_F$ is nontrivial.
Choose $f_F \in \ker{\pi_F}\backslash\{0\}$.
By continuity of measure we have
\[
     \lim_{\varepsilon \to 0+} \mu(\{x \in X : |f_F(x)| < \varepsilon\}) = \mu(\zz{f_F}) = 0.
\]
Choose $\varepsilon_F > 0$ such that
\[
	\mu(\{x \in X : |f_F(x)| < \varepsilon_F\})< \frac{1}{2}.
\]
Put 
\[
g_F=\frac{f_F}{\ep_F}\text{ and }U_F=\{x\in X: |g_F(x)|<1\} = g_F^{-1}((-1,1)).
\]
Then
\[
g_F |_{F} \equiv 0 \text{ and } \mu(U_F)< \frac 12.
\]
Put 
\[
Q = \{ g_F : \text{$F \subset X$ a finite nonempty subset}\}. 
\]

We show that there is no sequence $\{q_n\}_{n \in \N} \subset Q$ with $q_n \to 0$. Suppose the contrary: some sequence $\{q_n\}_{n \in \N} \subset Q$ converges to zero. 
For each $n \in \N$ there is a nonempty finite $F_n\subset X$ such that $q_n=g_{F_n}$.
Then $\liminf_{n\to\infty} U_{F_n} = X$. Since $\mu(U_{F_n})< \frac 12$ for $n\in\N$, Proposition~\ref{prop:pre:liminf-me} gives
\[
1=\mu(X)\leq \liminf_{n\to\infty} \mu\bigl( U_{F_n} \bigr) \leq \frac 12.
\]
A contradiction.

Since $g_F\in W(F,\ep)$ for any finite nonempty $F\subset X$ and $\ep>0$, we have $0\in\cl Q$.
Since there is no sequence $\{q_n\}_{n \in \N} \subset Q$ with $q_n \to 0$, the space $L$ is not Fr\'echet--Urysohn.
\end{proof}

A sequence of subsets $(M_n)_n$ of a set $X$ is point-finite if and only if 
$X=\liminf_{n\to \infty} X\setminus M_n$. Therefore Proposition~\ref{prop:M} implies the following statement.

\begin{theorem}
\label{th:M}
Let $X$ be a set and $L$ a subspace of $\R^X$. 
Then $L\in \dd_M$ if and only if $\liminf_{n\to\infty} \zz{f_n}\neq X$ for every sequence of nonzero functions from $L$.
\end{theorem}

Proposition~\ref{prop:FU} implies the following statement.

\begin{theorem}
\label{th:FU}
Let $X$ be a measurable space, $\mu$ a nonzero $\si$-finite measure on $X$, 
and $L\subset \R^X$ a subspace of the space $\cL^0(X,\mu)$ of $\mu$-measurable functions with the topology of pointwise convergence. If $\mu(\zz f)=0$ for every nonzero function $f$ in $L$, then $L\in\dd_{FU}$.
\end{theorem}

Let us consider an application of this theorem.

\begin{proposition}
\label{prop:a-FU}
Let $G \subset \R^n$ be a domain and let $\mathcal{A}(G)$ be the space of real-analytic functions on $G$ with the topology of pointwise convergence.
Then $\mathcal{A}(G)\in \dd_{FU}$.
\end{proposition}
\begin{proof}
By Proposition~1 of \cite{Mityagin2020}, the zero set of any nonzero function in $\mathcal{A}(G)$ has Lebesgue measure zero.
By Theorem~\ref{th:FU}, $\mathcal{A}(G)\in \dd_{FU}$.
\end{proof}

Let $X$ be a set, $L$ a subspace of $\R^X$ and $\fZ=\{\zz f: f\in L, f\neq 0\}$. Consider the following conditions on $L$:
\begin{enumerate}
\item[$(\cZ_1)$]
$\liminf_{n\to\infty} Z_n\neq X$ for every sequence $\{Z_n\in\fZ:n\in\N\}$;
\item[$(\cZ_2)$]
there is a nonzero $\si$-additive finite measure $\mu$ on some $\si$-algebra of subsets of $X$ such that $\mu(Z)=0$ for all $Z\in \fZ$ and every function from $L$ is $\mu$-measurable.
\end{enumerate}

Clearly, $(\cZ_2) \implies (\cZ_1)$.
Indeed, let $\mu$ be a measure from $(\cZ_2)$ and $Z_n\in\fZ$. Then $\bigcap_{n\ge k}Z_n\subset Z_k$ are $\mu$-null measurable sets, hence $\mu(\liminf_n Z_n)=0<\mu(X)$ and $\liminf_n Z_n\neq X$.

Theorems~\ref{th:M} and~\ref{th:FU} can be reformulated as follows.

\begin{proposition}
\label{prop:M-FU}
Let $X$ be a set and $L$ a subspace of $\R^X$. Then
\begin{enumerate}
\item[$\mathrm{M}:$]
$(\cZ_1) \lrarr L\in\dd_M$;
\item[$\mathrm{FU}:$]
$(\cZ_2) \rarr L\in\dd_{FU}$.
\end{enumerate}
\end{proposition}

\begin{question}\label{q:FU}
Find a necessary and sufficient condition on a subspace $L\subset\R^X$ for $L\in\dd_{FU}$ in $\mathrm{ZFC}$.
\end{question}
A criterion for $\dd_M$ is given by Theorem~\ref{th:M}; for $\dd_{FU}$ only the sufficient condition $(\mathcal Z_2)$ is known in $\mathrm{ZFC}$; in the Hru\v{s}\'ak--Ramos-Garc\'{\i}a model the criterion is $(\mathcal Z_1)$ (see Remark~\ref{note:Z3}).

\begin{proposition}\label{prop:fin-mult}
Let $\sB$ be a class of TVS which is closed under isomorphisms and hereditary with respect to closed subspaces (a closed subspace of any TVS from $\sB$ belongs to $\sB$). If $L\in\dd_\sB$, then $L\times\R\in\dd_\sB$.
\end{proposition}
\begin{proof}
Let $S\subset L\times\R$ be a subspace with $S\in\sB$, and let $\pi\colon S\to\R$ be the projection onto the second factor. Since $\{0\}$ is closed in $\R$, the subspace $S_1=\ker\pi=S\cap(L\times\{0\})$ is closed in $S$, hence $S_1\in\sB$. But $S_1$ is isomorphic to a subspace of $L$, hence $S_1$ is finite-dimensional. Since $S/S_1$ is linearly isomorphic to the subspace $\pi(S)\subset\R$, we get $\dim S\le\dim S_1+1<\infty$.
\end{proof}

The classes of metrizable TVS and of Fr\'echet--Urysohn TVS are closed under isomorphisms and are hereditary with respect to subspaces, in particular closed ones. Hence the classes $\dd_M$ and $\dd_{FU}$ are stable under multiplication by $\R$: if $L\in\dd_M$ (respectively, $L\in\dd_{FU}$), then $L\times\R\in\dd_M$ (respectively, $L\times\R\in\dd_{FU}$).

\begin{note}\label{note:plus}
Let $L\subset\R^X$ be an infinite-dimensional subspace belonging to $\dd_{FU}$. Such subspaces exist: for example, let $D$ be a countable discrete space, $X=C_p(D)=\R^D$ and $L=L_p(D)=(C_p(D))'\subset\R^X$; it follows from \cite[Theorem~7.10(i)]{LRS-TVS} that $L\in\dd_{FU}$ (see also Remark~\ref{note:Lp}). Let $x_*\notin X$, $X^+=X\cup\{x_*\}$ and $L^+=\{f\in\R^{X^+}: f|_X\in L\}$. The map $f\mapsto(f|_X,f(x_*))$ is a linear homeomorphism of $L^+$ onto $L\times\R$ (the topology of pointwise convergence on both sides); in particular, $L^+$ is infinite-dimensional. By Proposition~\ref{prop:fin-mult}, $L^+\in\dd_{FU}$. At the same time, condition $(\cZ_2)$ fails for $L^+$. Indeed, let $g\in L\setminus\{0\}$, and let $f_1,f_2\in L^+$ be given by $f_1|_X\equiv0$, $f_1(x_*)=1$ and $f_2|_X=g$, $f_2(x_*)=0$. Then $\zz{f_1}=X$, $\zz{f_2}=\zz g\cup\{x_*\}$, and $\zz{f_1}\cup\zz{f_2}=X^+$. If $\mu$ were a measure from $(\cZ_2)$ for $L^+$, then $\mu(X^+)\le\mu(\zz{f_1})+\mu(\zz{f_2})=0$, a contradiction.

It follows, first, that $(\cZ_2)$ does not follow from $(\cZ_1)$ ($L^+\in\dd_{FU}\subset\dd_M$, so $(\cZ_1)$ holds for $L^+$ by Theorem~\ref{th:M}). Second, condition $(\cZ_2)$ is not necessary for $L^+\in\dd_{FU}$. Hence $(\cZ_2)$ cannot be an answer to Question~\ref{q:FU}.
\end{note}

\section{The space of polynomials}\label{sec:pol}

Let $X\subset \R$ be infinite.
Denote by $\cP_X$ the space of polynomial functions on $X$ (restrictions to $X$ of polynomials $\R[t]$) with the topology of pointwise convergence, in other words
\[
\cP_X = \{ f|_X : f\in \R[t]\}\subset \R^X.
\]

Since the zero set of any nonzero element of $\cP_X$ is finite, the $\liminf$ of a countable family of such zero sets is countable and hence different from an uncountable $X$; that is, $(\mathcal Z_1)$ holds. Therefore Proposition~\ref{prop:M-FU} implies the following statement.

\begin{proposition}
\label{prop:pol:M}
If $X\subset\R$ is uncountable, then $\cP_X\in\dd_M$.
\end{proposition}

Proposition~\ref{prop:a-FU} implies $\cP_{\R}\in \dd_{FU}$. Below we refine this fact.

\begin{proposition}
\label{prop:pol:Z3}
Let $X\subset\R$ be infinite. Condition $(\mathcal Z_2)$ for $\cP_X$ holds if and only if $X$ is not of absolute measure zero.
\end{proposition}
\begin{proof}
$(\Leftarrow)$ If $X$ is not of absolute measure zero, then there is a nonzero continuous finite Borel measure $\mu$ on $X$. The zero sets of nonzero polynomials are finite, hence $\mu$-null; polynomials on $X$ are continuous, hence Borel and $\mu$-measurable. Therefore $(\mathcal Z_2)$ holds.

$(\Rightarrow)$ Let $\mu$ be a nonzero measure from condition $(\mathcal Z_2)$, defined on some $\si$-algebra $\Sigma$. The function $t\mapsto t$ on $X$ is $\Sigma$-measurable, and the Borel $\si$-algebra of $X$ is generated by this function, hence it is contained in $\Sigma$. Since $X$ is infinite, $t-x$ is a nonzero element of $\cP_X$, and its zero set is $\{x\}$ for $x\in X$; hence $\mu(\{x\})=0$. Thus the restriction of $\mu$ to the $\si$-algebra of Borel sets is a nonzero continuous Borel measure on $X$, that is, $X$ is not of absolute measure zero.
\end{proof}

\begin{proposition}
\label{prop:pol:FU}
If $X\subset\R$ is not of absolute measure zero, then $\cP_X\in\dd_{FU}$.
\end{proposition}
\begin{proof}
By Proposition~\ref{prop:pol:Z3}, condition $(\mathcal Z_2)$ holds for $\cP_X$. Proposition~\ref{prop:M-FU} (part FU) implies that $\cP_X\in\dd_{FU}$.
\end{proof}

\begin{cor}
\label{cor:pol:FU}
Let $X\subset\R$ and suppose that $\cP_X$ contains an infinite-dimensional Fr\'echet--Urysohn space.
Then $X$ is of absolute measure zero. 
\end{cor}

By (F1) and (F2) of Section~\ref{sec:pre}, if $X\subset\R$ and $C_p(X)$ is a Fr\'echet--Urysohn space, then $X$ is of absolute measure zero.
Corollary~\ref{cor:pol:FU} strengthens this statement in its hypothesis (instead of all of $C_p(X)$ it suffices to have an infinite-dimensional Fr\'echet--Urysohn subspace of $\cP_X$), but by the conclusion of (F2) a stronger statement holds: $X$ has strong measure zero.
Note also that if $X\subset\R$ and $C_p(X)$ is a $k_\R$-space, then $X$ is of absolute measure zero \cite[Theorem 11.18]{GR2025kr}.

\begin{example}
\label{ex:pol:FU}

Assume $\mathrm{MA+\lnot CH}$.
Let $X\subset \R$, $|X|=\om_1$. By (F3), $X$ has the property $(\gamma)$ and hence $C_p(X)$ is a Fr\'echet--Urysohn space, nonmetrizable since $X$ is uncountable. The subspace $\cP_X$ of the Fr\'echet--Urysohn space $C_p(X)$ is a Fr\'echet--Urysohn space. Since polynomials interpolate functions on finite sets, the functionals $\de_x$, $x\in X$, are linearly independent in $\cP_X'$, and $\cP_X'$ is uncountable-dimensional; by Lemma~\ref{l:L-metr-1} the space $\cP_X$ is nonmetrizable. Proposition~\ref{prop:pol:M} implies that $\cP_X\in\dd_M$.

The topology of $\cP_X$ is weak: by Lemma~\ref{l:L-dual} the space $\cP_X'$ is generated by the functionals $\de_x$, $x\in X$, hence the topology of pointwise convergence coincides with the weak topology $\sigma(\cP_X,\cP_X')$. Consequently, under $\mathrm{MA+\lnot CH}$, for the LCS $V=\cP_X$ we have $V=V_w$ and $V_w\in\dd_M\setminus\dd_{FU}$: the equality $\dd_M=\dd_{FU}$ fails even for spaces with the weak topology.
\end{example}

\begin{note}\label{note:Z3}
\textbf{(On the limits of applicability of Theorem~\ref{th:FU}.)}
For $\cP_X$, $X\subset\R$, condition $(\mathcal Z_2)$ is equivalent to $X$ not being of absolute measure zero (Proposition~\ref{prop:pol:Z3}). Two conclusions follow.
\begin{enumerate}[leftmargin=0pt,itemindent=2em,labelsep=0.5em]
\item In the Hru\v{s}\'ak--Ramos-Garc\'{\i}a model $\dd_M=\dd_{FU}$ (Theorem~\ref{th:ddM-ddFU}), hence for every subspace $L\subset\R^X$ condition $(\mathcal Z_1)$ is equivalent to $L\in\dd_{FU}$ (Proposition~\ref{prop:M-FU}); that is, Theorem~\ref{th:M} gives a complete description of $\dd_{FU}$ in this model. At the same time, condition $(\mathcal Z_2)$ is not necessary for $\dd_{FU}$: in $\mathrm{ZFC}$ there is an uncountable set $X\subset\R$ of absolute measure zero \cite[Theorem~8.3]{Bukovsky}; for it $\cP_X\in\dd_M$ by Proposition~\ref{prop:pol:M}, hence $\cP_X\in\dd_{FU}$ in the Hru\v{s}\'ak--Ramos-Garc\'{\i}a model, but $(\mathcal Z_2)$ does not hold.
\item Under the assumptions of Example~\ref{ex:pol:FU}, $(\mathcal Z_1)$ holds for $\cP_X$ (zero sets of nonzero polynomials are finite and $X$ is uncountable), but $\cP_X$ is an infinite-dimensional Fr\'echet--Urysohn space, that is, $\cP_X\notin\dd_{FU}$. Therefore, under $\mathrm{MA+\lnot CH}$, condition $(\mathcal Z_2)$ in the sufficient condition of Proposition~\ref{prop:M-FU} cannot be replaced by $(\mathcal Z_1)$.
\end{enumerate}
\end{note}

\begin{question}\label{q:ex}
Is there in $\mathrm{ZFC}$ an infinite set $X\subset\R$ of absolute measure zero such that $\cP_X\in\dd_{FU}$? Is there such an $X$ under $\mathrm{MA+\lnot CH}$?
\end{question}
In the Hru\v{s}\'ak--Ramos-Garc\'{\i}a model such an $X$ exists (conclusion~1 of Remark~\ref{note:Z3}). Under $\mathrm{MA+\lnot CH}$ it must have cardinality $\mathfrak c$: infinite $X$ of smaller cardinality are either countable or $\gamma$-sets, and then $\cP_X\notin\dd_{FU}$.

\section{The classes $\dd_M$ and $\dd_{FU}$ in locally convex spaces with the weak topology}
\label{sec:wlcs}

Let $V$ be an LCS, $V_w$ the space $V$ with its weak topology and $V'_{w^*}$ the dual space with its weak${}^\ast$ topology.

For $i=1,2$ we define a class $\cW_i$ of LCS: $V\in \cW_i$ if condition $\cZ_i$ holds for $X=V'$ and $L=c(V)$. That is,
\begin{enumerate}
\item[$(\cW_1)$]
$\liminf_{n\to\infty} \zz{\cc v_n}\neq V'$ for every sequence $\{v_n\in V\setminus\{0\}:n\in\N\}$;
\item[$(\cW_2)$]
there is a nonzero $\si$-additive finite measure $\mu$ on some $\si$-algebra of subsets of $V'$ such that $\mu(\zz{\cc v})=0$ for all $v\in V\setminus\{0\}$ and every function from $c(V)$ is $\mu$-measurable.
\end{enumerate}

Proposition~\ref{prop:M-FU} implies the following statement.

\begin{proposition}
\label{prop:wlcs:M-FU}
Let $V$ be an LCS.
Then
\begin{enumerate}
\item[$\mathrm{M}:$]
$V\in \cW_1 \lrarr V_w\in\dd_M$;
\item[$\mathrm{FU}:$]
$V\in \cW_2 \rarr V_w\in\dd_{FU}$.
\end{enumerate}
\end{proposition}

\begin{proposition}
\label{prop:wlcs:cbi}
Let $V$ and $U$ be LCS and $\varphi\colon V\to U$ a continuous linear bijection.
For $i=1,2$, if $U\in \cW_i$, then $V\in \cW_i$.
\end{proposition}
\begin{proof}
Since $\varphi$ is continuous and linear, the map
\[
\varphi^{\#}\colon U'\to V',\ \te\mapsto  \te \circ \varphi
\]
is well defined. Since $\varphi(V)=U$, $\varphi^{\#}$ is injective. For $v\in V$ and $\te\in U'$ we have
\[
\cc v(\varphi^{\#}\te)=\te(\varphi(v))=\overline{\varphi(v)}(\te),
\]
that is,
\begin{equation}\label{eq:cbi}
(\varphi^{\#})^{-1}\bigl(\zz{\cc v}\bigr)=\zz{\overline{\varphi(v)}}.
\end{equation}
Let $\{v_n\in V\setminus\{0\}:n\in\N\}$. Put $u_n=\varphi(v_n)\neq 0$ for $n\in\N$. Put $A=\varphi^{\#}(U')\subset V'$. Since $\varphi^{\#}$ is injective, \eqref{eq:cbi} implies that
\[
A\cap \zz{\cc v_n}=\varphi^{\#}\bigl(\zz{\cc u_n}\bigr).
\]
The operation $B\mapsto A\cap B$ and the injective map $\varphi^{\#}$ commute with $\liminf$, so
\begin{equation}\label{eq:cbi2}
A\cap \liminf_{n\to\infty} \zz{\cc v_n} = \varphi^{\#}\Bigl(\liminf_{n\to\infty} \zz{\cc u_n}\Bigr).
\end{equation}

Suppose $U\in\cW_1$. Then $\liminf_{n\to\infty} \zz{\cc u_n}\neq U'$. Since $\varphi^{\#}$ is injective, the right-hand side of \eqref{eq:cbi2} is not equal to $\varphi^{\#}(U')=A$. Hence $A\not\subset\liminf_{n\to\infty} \zz{\cc v_n}$, and therefore $\liminf_{n\to\infty} \zz{\cc v_n}\neq V'$. Thus $V\in\cW_1$.

Suppose $U\in\cW_2$, and let $\nu$ be a nonzero $\si$-additive finite measure on a $\si$-algebra $\Sigma$ of subsets of $U'$ from the definition of $\cW_2$. Put
\[
\Sigma_V=\{B\subset V':(\varphi^{\#})^{-1}(B)\in\Sigma\},\qquad \mu(B)=\nu\bigl((\varphi^{\#})^{-1}(B)\bigr),\ B\in\Sigma_V.
\]
Then $\Sigma_V$ is a $\si$-algebra, $\mu$ is a finite $\si$-additive measure on $\Sigma_V$ and $\mu(V')=\nu(U')>0$. Let $v\in V\setminus\{0\}$. By \eqref{eq:cbi},
\[
\mu(\zz{\cc v})=\nu\bigl(\zz{\overline{\varphi(v)}}\bigr)=0,
\]
since $\varphi(v)\neq 0$. For every Borel $O\subset \R$ we have
\[
(\varphi^{\#})^{-1}\bigl(\cc v^{-1}(O)\bigr)=\overline{\varphi(v)}^{-1}(O)\in\Sigma,
\]
since the function $\overline{\varphi(v)}$ is $\nu$-measurable. Hence $\cc v$ is $\mu$-measurable. Thus $V\in\cW_2$.
\end{proof}

\begin{proposition}
\label{prop:wlcs:norm-w1}
If $V$ is a normable space, then $V\in \cW_1$.
\end{proposition}
\begin{proof}
Fix a norm on $V$. Then $V'$ is a Banach space with the dual norm. Let $\{v_n\in V\setminus\{0\}:n\in\N\}$. The set $H_n=\zz{\cc v_n}=\{\te\in V':\te(v_n)=0\}$ is a closed linear subspace of $V'$, since the functional $\te\mapsto\te(v_n)$ is continuous. By the Hahn--Banach theorem there is $\te\in V'$ with $\te(v_n)\neq 0$, hence $H_n\neq V'$. A proper closed subspace of a normed space is nowhere dense (it contains no ball). By the Baire category theorem $\bigcup_{n=1}^\infty H_n\neq V'$. Since $\liminf_{n\to\infty}H_n\subset\bigcup_{n=1}^\infty H_n$, we get $\liminf_{n\to\infty}H_n\neq V'$. Hence $V\in\cW_1$.
\end{proof}

Propositions~\ref{prop:wlcs:cbi} and~\ref{prop:wlcs:norm-w1} imply that if an LCS $V$ admits a continuous norm $p$, then $V\in\cW_1$ (the identity map $V\to(V,p)$ is a continuous linear bijection). By Proposition~\ref{prop:wlcs:M-FU}, $V_w\in \dd_M$.

\begin{lemma}
\label{lem:wlcs:sep-norm-w2}
If $V$ is a separable normable space, then $V\in\cW_2$; that is, 
there is a Borel probability measure $\mu$ on $V'_{w^*}$ such that for all $v \in V\setminus\{0\}$
	$\mu(\{\te \in V' : \te(v) = 0\}) = 0$.
\end{lemma}
\begin{proof}
		Let $\{v_n\}_{n \in \N}$ be a dense subset of $V$. By the Hahn--Banach theorem, for each $n\in \N$ there is a functional $f_n \in V'$ with $\|f_n\| = 1$ and $f_n(v_n) = \|v_n\|$. The sequence $\{f_n\}_{n \in \N}$ separates points: if $v \in V$ and $f_n(v) = 0$ for all $n$, choose a sequence $v_{n_k} \to v$; then
		\[
\|v_{n_k}\| = f_{n_k}(v_{n_k}) = f_{n_k}(v_{n_k} - v) \leq \|v_{n_k} - v\| \to 0.
\]
		Hence $v_{n_k} \to 0$, and therefore $v = 0$.
		For each $n \in \N$ put $\Omega_n = [-1,1]$, $\mathcal{A}_n = \mathcal{B}([-1, 1])$, and let $\lambda_n$ be the normalized Lebesgue measure on $[-1, 1]$. Put
		\[
		(\Omega, \mathcal{A}, \mathbb{P}) = \bigotimes_{n \in \N}(\Omega_n, \mathcal{A}_n, \lambda_n),
		\]
		the countable product of probability spaces, endowed with the product topology. Let $\xi_n : \Omega \rightarrow [-1,1]$ be the coordinate projections. They are continuous, hence they are random variables. Define a random linear Borel functional $F: \Omega \rightarrow V'_{w^*}$ by
		$$F(\omega) = \sum_{n=1}^{\infty}\dfrac{ \xi_n(\omega)f_n}{2^n}.$$
		Let us show that this definition is correct: for each $\omega \in \Omega$ the series converges in $V'$, and the map $F$ is continuous in the weak$^*$ topology (hence preimages of open sets are Borel). Indeed:
		\begin{enumerate}
		\item For each $\omega \in \Omega$ the series converges absolutely in the norm of $V'$, since
		$$\sum_{n=1}^{\infty}\left\|\dfrac{\xi_n(\omega)f_n}{2^n}\right\| = \sum_{n=1}^{\infty}\dfrac{|\xi_n(\omega)|\,\|f_n\|}{2^n} \leq \sum_{n=1}^{\infty}\dfrac{1}{2^n} = 1 < \infty.$$
		The space $V'$ is complete, so the series converges in the norm of $V'$, and $F(\omega)\in V'$ is a continuous linear functional with $\|F(\omega)\| \leq 1$; that is, $F(\Omega)$ lies in the unit ball of $V'$.
		\item The partial sums $\sum\limits_{n=1}^{N}\dfrac{ \xi_n(\omega)f_n}{2^n}$ are continuous as maps from $\Omega$ to $V'$, and the series converges uniformly on $\Omega$, since
		$$\left\|\sum\limits_{n=N}^{\infty}\dfrac{ \xi_n(\omega)f_n}{2^n}\right\| \leq \sum\limits_{n=N}^{\infty}\dfrac{1}{2^n} = 2^{1-N}\to 0\quad (N\to\infty).$$
		Hence $F\colon \Omega \rightarrow V'$ is continuous as a uniform limit of continuous maps. Then $F$ is also continuous in the weaker topology, as a map $F: \Omega \rightarrow V'_{w^*}$.
		\end{enumerate}
		Let $\mu$ be the distribution of the random functional, which is a Borel probability measure on $V'_{w^*}$. By the estimates above, this measure is concentrated on the unit ball. We show that $\mu$ is the required measure. (Each functional $\cc v$, $v\in V$, is weak$^*$ continuous, hence Borel, that is, $\mu$-measurable.)
		Let $v \in V\setminus\{0\}$; denote $A_v = \{\te \in V' : \te(v) = 0\}$. This is a closed set, hence measurable. We get:
		$$\mu(A_v) = \mathbb{P}(F^{-1}(A_v)) = \mathbb{P}\Bigl(\{\omega \in \Omega : (F(\omega))(v) = 0\}\Bigr).$$
		Denote $E = \{\omega \in \Omega : (F(\omega))(v) = 0\}$; it suffices to show that $\mathbb{P}(E) = 0$.\\
		Since the sequence $\{f_n\}_{n \in \N}$ separates points of $V$, there is $k$ with $f_k(v) \neq 0$. Represent $\Omega$ as a product over the $k$-th coordinate and the remaining ones:
		$$\Omega = \Omega_k \times \Omega', \quad \Omega' = \prod_{n \neq k}\Omega_n, \quad \mathbb{P} = \lambda_k \otimes \mathbb{P}', \quad \mathbb{P}' = \bigotimes_{n \neq k}\lambda_n.$$
	We write a point $\omega \in \Omega$ as a pair $\omega = (t, \omega')$, where $t = \xi_k(\omega) \in [-1,1]$, $\omega' \in \Omega'$. Separating the $k$-th term, we get
		$$(F(\omega))(v) = \dfrac{tf_k(v)}{2^k}+ \sum_{n \neq k} \dfrac{\xi_n(\omega')\, f_n(v)}{2^n}.$$
		The indicator function $(t, \omega') \mapsto \chi_E(t, \omega')$ belongs to $L^1$, hence Fubini's theorem applies:
		\begin{align*}
		\mathbb{P}(E) &= \int_{\Omega_k \times \Omega'} \chi_E(t, \omega')\,d(\lambda_k \otimes \mathbb{P}')\\
		&= \int_{\Omega'}\Bigl(\int_{\Omega_k} \chi_E(t, \omega')\, \lambda_k(dt)\Bigr) \mathbb{P}'(d\omega') = 0.
		\end{align*}
		The last equality follows from the fact that for each fixed $\omega'$ the function $t\mapsto (F(t, \omega'))(v)$ vanishes at a single point
		\[
		t^*=-\dfrac{2^k}{f_k(v)}\sum_{n \neq k} \dfrac{\xi_n(\omega')\, f_n(v)}{2^n},
		\]
		and hence $\int\limits_{\Omega_k} \chi_E(t, \omega')\, \lambda_k(dt)  = 0$.
\end{proof}

\begin{theorem}
\label{th:wlcs:norm-dd-FU}
If an LCS $V$ admits a continuous norm, then $V_w\in \dd_{FU}$.
\end{theorem}
\begin{proof}
Suppose the contrary. Then $V_w$ contains an infinite-dimensional Fr\'echet--Urysohn subspace $S$. Choose a countable-dimensional infinite-dimensional subspace $W\subset S$. As a subspace of a Fr\'echet--Urysohn space, $W$ (with the topology induced from $V_w$) is a Fr\'echet--Urysohn space. Since $V$ is locally convex, by the Hahn--Banach theorem every continuous functional on $W$ extends to $V$, so the topology induced on $W$ from $V_w$ coincides with the weak topology $W_w$ of the space $W$. Thus $W_w$ is an infinite-dimensional Fr\'echet--Urysohn space.

Let $p$ be a continuous norm on $V$. Denote by $U$ the space $W$ with the norm $p|_W$. Then the identity map $\varphi\colon W\to U$ is a continuous linear bijection onto the normed space $U$. The space $U$ is countable-dimensional, hence separable (linear combinations of elements of a Hamel basis with rational coefficients are dense in $U$). By Lemma~\ref{lem:wlcs:sep-norm-w2} there is a Borel probability measure $\mu$ on $U'_{w^*}$ such that $\mu(\zz{\cc u})=0$ for all $u\in U\setminus\{0\}$. Each function $\cc u$ is weak$^*$ continuous, hence Borel, that is, $\mu$-measurable. Hence $U\in\cW_2$. Proposition~\ref{prop:wlcs:cbi} implies that $W\in\cW_2$, and Proposition~\ref{prop:wlcs:M-FU} implies that $W_w\in\dd_{FU}$. But $W_w$ is an infinite-dimensional Fr\'echet--Urysohn space. A contradiction.
\end{proof}

\begin{cor}
\label{cor:wlcs:norm-dd-FU}
If an LCS $V$ admits a continuous norm and $V_w$ is Fr\'echet--Urysohn, then $V$ is finite-dimensional.
\end{cor}

\begin{example}\label{ex:Vtau}
Let $(E,\|\cdot\|)$ be an infinite-dimensional normed space (for example, $\ell_2$), $E^{\#}$ the space of all (not only continuous) linear functionals on $E$, and $V$ the space $E$ with the topology generated by the norm and all functionals from $E^{\#}$. Then $V$ is an LCS with the continuous norm $\|\cdot\|$, all linear functionals on $V$ are continuous, that is, $V'=E^{\#}$, and $V_w=(E,\sigma(E,E^{\#}))$. By Theorem~\ref{th:wlcs:norm-dd-FU}, $V_w\in\dd_{FU}$. This inclusion can also be obtained differently. Let $X$ be a discrete space of cardinality $\dim E$. Choosing a Hamel basis of $E$, we get a linear isomorphism $E\cong\R^{(X)}$ (the space $L(X)$ as a vector space) and $E^{\#}\cong\R^X=C(X)$; hence $V_w=(E,\sigma(E,E^{\#}))$ is isomorphic to $L_p(X)$, and $V_w\in\dd_{FU}$ by Remark~\ref{note:Lp}. The space $V$ is not normable: if the topology of $V$ were given by a norm $q$, then all linear functionals on $(E,q)$ would be continuous, whereas on an infinite-dimensional normed space there is a discontinuous linear functional. Moreover, $V_w$ is not isomorphic to $Y_w$ for any normable $Y$: an isomorphism $T\colon V_w\to Y_w$ would give $Y'=Y^{\#}$, that is, continuity of all linear functionals on $Y$, whence $Y$ is finite-dimensional. Consequently, Theorem~\ref{th:wlcs:norm-dd-FU} does not reduce to the case of normable spaces.
\end{example}

\begin{note}\label{note:Lp}
The existence of a continuous norm in Theorem~\ref{th:wlcs:norm-dd-FU} is sufficient but not necessary for $V_w\in\dd_{FU}$. Let $V=L(X)$ be the free LCS over a Tychonoff space $X$. Then $V_w=L_p(X)$ (the space $L(X)$ with its weak topology; the dual of $L(X)$ is $C(X)$, so $L_p(X)\subset\R^{C(X)}$), and by \cite[Theorem~7.10(i)]{LRS-TVS} every $\varkappa$-Fr\'echet--Urysohn (in particular, Fr\'echet--Urysohn) linear subspace of $L_p(X)$ is finite-dimensional, that is, $V_w\in\dd_{FU}$. At the same time, if $X$ is not submetrizable (for example, $X=\beta\N$), then $L(X)$ has no continuous norm: a continuous norm on $L(X)$ would define a continuous metric $d(x,y)=\|x-y\|$ on $X$, that is, $X$ would be submetrizable.
\end{note}

\section{The dual space with the weak\texorpdfstring{$^*$}{*} topology}\label{sec:wstar}

Let $V$ be an LCS and let $V'_{w^*}$ be the dual space $V'$ with the weak$^*$ topology $\sigma(V',V)$. The space $V'_{w^*}$ is naturally embedded into $\R^V$ (a functional $f\in V'$ is regarded as a function $f\colon V\to\R$), and its topology is the topology of pointwise convergence on $V$.

\begin{theorem}\label{th:wstar-complete}
Let an LCS $V$ admit a complete norm $\|\cdot\|$ whose topology is finer than the given topology of $V$ (that is, the identity map $(V,\|\cdot\|)\to V$ is continuous). In particular, this is the case if $V$ is Banach. Then $V'_{w^*}\in\dd_{FU}$.
\end{theorem}
\begin{proof}
Suppose the contrary. Then $V'_{w^*}$ contains an infinite-dimensional Fr\'echet--Urysohn subspace, and it contains a countable-dimensional infinite-dimensional subspace $L_0$, which is also Fr\'echet--Urysohn. Write $L_0=\bigcup_k L_k$, where $L_1\subset L_2\subset\dots$ are finite-dimensional subspaces. For each $k$ there is a finite set $F_k\subset V$ such that the map $L_k\to\R^{F_k}$, $f\mapsto f|_{F_k}$, is injective (choose $\dim L_k$ points at which the determinant built from a basis of $L_k$ is nonzero). Let $Z$ be the closure of the linear span of $\bigcup_k F_k$ in the Banach space $(V,\|\cdot\|)$. The space $Z$ is separable, and $f|_Z\neq0$ for every $f\in L_0\setminus\{0\}$, since $f\in L_k$ for some $k$, and hence $f|_{F_k}\neq0$.

Let $\{z_n:n\in\N\}$ be dense in $Z\setminus\{0\}$ (with respect to $\|\cdot\|$), let $\xi_n$ be independent random variables uniformly distributed on $[-1,1]$, and let
\[
\Phi=\sum_{n=1}^\infty \frac{\xi_n z_n}{2^n \|z_n\|}\in Z .
\]
The series converges uniformly, so $\Phi$ is a random element of $(Z,\|\cdot\|)$. Put $\Sigma=\{A\subset V:\ A\cap Z\ \text{is Borel in}\ (Z,\|\cdot\|)\}$ and $\mu(A)=\mathbf P(\Phi\in A)$ for $A\in\Sigma$. Then $\mu$ is a probability measure on the $\si$-algebra $\Sigma$. Every function $f\in L_0$ is continuous on $V$, hence on $(V,\|\cdot\|)$; therefore $f^{-1}(O)\cap Z$ is Borel in $Z$ for every Borel $O\subset\R$, that is, $f$ is $\mu$-measurable.

Let $f\in L_0\setminus\{0\}$. Since $f|_Z\neq0$ is continuous on $Z$, there is $k$ with $f(z_k)\neq0$. Fix the values of all $\xi_n$, $n\neq k$. Then $f(\Phi)=\alpha+\xi_k\,f(z_k)/(2^k\|z_k\|)$ is an affine function of $\xi_k$ with a nonzero coefficient, and it vanishes for at most one value of $\xi_k$. By Fubini's theorem $\mu(\zz f)=\mathbf P(f(\Phi)=0)=0$. Thus condition $(\cZ_2)$ holds for $L_0\subset\R^V$, and by Proposition~\ref{prop:M-FU} $L_0\in\dd_{FU}$. But $L_0$ is infinite-dimensional and Fr\'echet--Urysohn. A contradiction.
\end{proof}

For Banach $V$, Theorem~\ref{th:wstar-complete} is proved in \cite[Theorem~5.8]{LRS-TVS} (where it is also shown that $V'_{w^*}$ has no infinite-dimensional $\varkappa$-Fr\'echet--Urysohn or Baire subspaces); the proof above uses measure theory and is a variant of the method of Theorem~\ref{th:wlcs:norm-dd-FU}.

\begin{note}\label{note:wstar-countable}
The completeness of the norm in Theorem~\ref{th:wstar-complete} cannot be dropped. Let $V$ be a normed space of countably infinite dimension and let $\{e_k:k\in\N\}$ be its Hamel basis. The topology of $V'_{w^*}$ is generated by the countable family of functionals $f\mapsto f(e_k)$, hence $V'_{w^*}$ is metrizable; it is infinite-dimensional, since $V'$ separates the points of $V$. Hence $V'_{w^*}\notin\dd_M\supset\dd_{FU}$ (cf.\ \cite[Example~5.9]{LRS-TVS}).
\end{note}

\subsection*{Declaration of competing interest}
The authors declare that they have no known competing financial interests or personal relationships that could have appeared to influence the work reported in this paper.

\subsection*{Data availability}
No data were used for the research described in the article.


\end{document}